\UseRawInputEncoding
\documentclass[10pt, reqno]{amsart}
\numberwithin{equation}{section}
\usepackage{color, cite}
\usepackage{amssymb}
\usepackage{hyperref}

\newcommand{\qtq}[1]{\quad\text{#1}\quad}

\theoremstyle{definition}
\newtheorem{definition}{Definition}

\theoremstyle{plain}
\newtheorem{theorem}[definition]{Theorem}
\newtheorem*{theorem*}{Theorem}
\newtheorem{lemma}{Lemma}
\newtheorem{proposition}[definition]{Proposition}
\newtheorem*{claim*}{Claim}
\newtheorem{corollary}[definition]{Corollary}

\theoremstyle{remark}
\newtheorem*{remark*}{Remark}

\newcommand{\eps}{\varepsilon}

\DeclareMathOperator{\R}{\mathbb{R}}

\DeclareMathOperator{\C}{\mathbb{C}}
\DeclareMathOperator{\F}{\mathcal{F}}

\newcommand{\lspace}[2]{L^{#1}_{#2}}

\newcommand{\inv}[1]{{#1}^{-1}}

\newcommand{\jbrak}[1]{\langle#1\rangle}

\begin{document}

\title[Modified Scattering]{Modified scattering for a dispersion-managed nonlinear Schr\"odinger equation}
\author[J. Murphy]{Jason Murphy}
\address{Department of Mathematics \& Statistics, Missouri S\&T}
\email{jason.murphy@mst.edu}
\author[T. Van Hoose]{Tim Van Hoose}
\address{Department of Mathematics \& Statistics, Missouri S\&T}
\email{trvkdb@mst.edu}

\begin{abstract}
We prove sharp $L^\infty$ decay and modified scattering for a one-dimensional dispersion-managed cubic nonlinear Schr\"odinger equation with small initial data chosen from a weighted Sobolev space.  Specifically, we work with an averaged version of the dispersion-managed NLS in the strong dispersion management regime.  The proof adapts techniques from \cite{HN, KP}, which established small-data modified scattering for the standard $1d$ cubic NLS.
\end{abstract}
    \maketitle

\section{Introduction}
We study the long-time behavior of small solutions for a `dispersion-managed' nonlinear Schr\"odinger equation (NLS) in one space dimension.  Such models arise as the envelope equation for electromagnetic wave propagation in fiber-optics communication systems in which the dispersion is varied periodically along the optical fiber (see e.g. \cite{Agr, CHL, Kur}).  This may be modeled via an equation of the form
\begin{equation}\label{dmnls0}
i\partial_t w + d(t) \partial_{xx} w = c|w|^2 w
\end{equation}
for some $1$-periodic function $d(t)$ and some $c\in\R\backslash\{0\}$.  One may then decompose 
\[
d(t)=d_{\text{av}}+ d_0(t),
\]
with $d_0$ 1-periodic and mean zero, and $d_{\text{av}}$ giving the average dispersion. 

In the present paper, we will not work directly with the model \eqref{dmnls0}, but rather with an `averaged' version in the so-called strong dispersion management regime (as introduced in \cite{GT1}).  In particular, we consider the equation
\begin{equation}\label{dmnls}
i\partial_t u + d_{\text{av}}\partial_{xx} u = c\int_0^1 e^{-iD(\tau)\Delta}\bigl\{ |e^{iD(\tau)\Delta}u|^2 e^{iD(\tau)\Delta}u\bigr\}\,d\tau,
\end{equation}
where
\[
D(\tau):= \int_0^\tau d_0(\sigma)d\sigma
\]
(see e.g. \cite[Section~1.2]{CHL} for a derivation of this model).  Here we work with the specific choice 
\begin{equation}
    d(t) = 
\begin{cases}
    d_+, \quad &0 \leq t < t_+, \\
    -d_-, \quad &t_+ < t \leq 1, \\
\end{cases}    
\end{equation}
with $d_\pm$ and $t_+$ chosen so that $d_{\text{av}}\neq 0$.  As a matter of fact, because we consider only small initial conditions, the sign of $d_{\text{av}}$ and $c$ play no essential role, and so we assume without loss of generality that $d_{\text{av}}=c=1$. 

Dispersion-managed nonlinear Schr\"odinger equations have been the subject of a great deal of recent research, due largely to their connection with applications in fiber-optics based communications.  This includes both numerical investigations and rigorous mathematical studies (see e.g. \cite{AntSauSpa, CHL, EHL, GT1, GT2, HL, HL2, PZ, ZGJT}).  Much of the work to date has centered on questions of well-posedness and the existence and properties of soliton solutions.  In this work, we will show that for small initial data chosen from a weighted Sobolev space, the corresponding solutions to \eqref{dmnls} are global in time and decay as $|t|\to\infty$.  Moreover, such solutions exhibit modified scattering, that is, asymptotically linear behavior up to a logarithmic phase correction (as in the case of the standard $1d$ cubic NLS).  In particular, our result demonstrates the absence of small coherent structures for \eqref{dmnls}.

Our main result is the following. 

\begin{theorem}\label{thm1}
Let $u_0\in H^{1,1}$ satisfy $\|u_0\|_{H^{1,1}}=\eps>0$.  If $\eps$ is sufficiently small, then there exists a unique solution $u\in C_t H_x^{1,1}([0,\infty)\times\R)$ to \eqref{dmnls} with $u|_{t=0}=u_0$.  Furthermore, the solution obeys
\begin{equation}\label{u-asymptotic}
\|u(t)\|_{L^\infty} \lesssim \eps(1+|t|)^{-\frac12}
\end{equation}
for all $t\geq 0$, and there exists $W\in L^\infty(\R)$ such that
\[
u(t,x) = (2it)^{-1/2} e^{ix^2/4t}\bigl[ \exp\{-\tfrac{i}2|W(\tfrac{x}{2t})|^2\log t\}  W(\tfrac{x}{2t})\bigr] + \mathcal{O}(t^{-\frac{1}{2}-\frac{1}{20}}) 
\]
in $L^\infty$ as $t\to\infty$. 
\end{theorem}

Theorem~\ref{thm1} fits in the general context of modified scattering for long-range nonlinear Schr\"odinger equations.  In particular, many previous works have considered the standard $1d$ cubic NLS
\begin{equation}\label{cubic-NLS}
i\partial_t u + \partial_{xx}u = \pm |u|^2 u,
\end{equation}
for which one also obtains sharp $L^\infty$ and modified scattering for small data in $H^{1,1}$.  In fact, in the defocusing case, one can capitalize on the completely integrable structure of \eqref{cubic-NLS} to obtain this result without any size restriction on the initial data \cite{DeiftZhou}.  Several different approaches have been utilized in order to establish small-data modified scattering for \eqref{cubic-NLS} (see e.g. \cite{HN, KP, LS, IT}, as well as \cite{Murphy} for a review). Essentially, each of these approaches are based off of a bootstrap argument involving some `dispersive' type norm and some `energy' type norm, using an ODE argument to obtain estimates for the dispersive part and a chain-rule type estimate and Gr\"onwall to control the energy part.  We follow the same general strategy, adapting techniques particularly from \cite{KP} and \cite{HN}.  In particular, we use the Fourier representation and a `space-time non-resonance' type approach as in \cite{KP} to control the dispersive norm, while the energy-type estimate relies on the chain-rule type identity satisfied by the Galilean operator $J(t)=x+2it\partial_x$.  The estimate for the dispersive norm follows largely as in \cite{KP}, while the estimate of energy norm requires some modifications.  In particular, while the energy estimate for the standard NLS relies directly on the $L^\infty$-decay for solutions, we must instead rely on the factorization of the free propagator (see \eqref{MDFM}) in order to exhibit suitable decay in the nonlinearity.  For the details, see \eqref{energy-pause} and \eqref{new-estimate} below.  

The problem for \eqref{dmnls} (as opposed to the original model \eqref{dmnls0}) is simplified by the fact that the underlying linear part of the equation is still given by the standard Schr\"odinger equation.  In the setting of \eqref{dmnls0}, the situation is more complicated due to the time-dependence in the linear part of the equation, which affects the underlying linear dispersion (see e.g. \cite{AntSauSpa}).  We plan to study the model \eqref{dmnls0} in a future work. 

The rest of the paper is organized as follows:  In Section~\ref{S:notation}, we collect some notation, a few useful identities and inequalities, and discuss the $H^1$ and $H^{1,1}$ well-posedness for \eqref{dmnls}.  In Section~\ref{S:proof}, we prove the main result, Theorem~\ref{thm1}.  In particular, in Section~\ref{S:decay}, we establish global existence and $L^\infty$ decay, while in Section~\ref{S:scattering} we establish the long-time asymptotic behavior of solutions.

\section{Preliminaries}\label{S:notation}
In this section, we introduce some notation that will be used throughout the rest of the paper. First, we write $A \lesssim B$ or $B \gtrsim A$ to denote the inequality $A \leq CB$ for some $C > 0$. If $A \lesssim B$ and $B \lesssim A$, we write $A \sim B$.  We also utilize the standard `big oh' notation $\mathcal{O}$. Finally, we use the Japanese bracket notation $\jbrak{x} := (1 + |x|^2)^{\frac{1}{2}}$ 

Our notation for the Fourier transform is
\[
  \F[f](\xi) = \hat{f}(\xi) := (2\pi)^{-\frac{1}{2}} \int_{\R} e^{i x  \xi} f(x) \,dx
\]
We define the spaces $H^1$ and $H^{1,1}$ via the norms
\[
\|u\|_{H^1} = \|\langle \partial_x\rangle u\|_{L^2} \qtq{and}\|u\|_{H^{1,1}} = \|u\|_{H^1}+\|xu\|_{L^2},
\]
where $\langle \partial_x\rangle = \F^{-1}\langle \xi\rangle \F$. 

Solutions to the linear Schr\"odinger equation with $u|_{t=0}=u_0$ are given by 
\[
u(t, x) = e^{it\Delta}u_0(x), \qtq{where} \F^{-1}e^{-it\xi^2}\F.
\]
The free propagator $e^{it\Delta}$ admits the integral kernel
\[
e^{it\Delta}(x,y) := (4\pi i t)^{-\frac{1}{2}} e^{i\frac{(x-y)^2}{4t}},
\]
which implies the factorization identity
\begin{equation}\label{MDFM}
e^{it\Delta} = \mathcal{M}(t)\mathcal{D}(t)\F \mathcal{M}(t),
\end{equation}
where 
\[
    [\mathcal{M}(t)f](x) = e^{ix^2/4t}f(x)\qtq{and}  [\mathcal{D}(t)f](x) = (2it)^{-\frac{1}{2}}f\bigl(\tfrac{x}{2t}\bigr)
\]

We will make use of the Galilean operator $J(t) = x + 2it\partial_x$. On the one hand, one can directly compute and show that
\[
J(t) = \mathcal{M}(t)[2it\partial_x] \mathcal{M}(-t).    
\]
On the other hand, an ODE argument leads to the identity 
\begin{equation}\label{J-id}
J(t) = e^{it\Delta}xe^{-it\Delta}.    
\end{equation}

We will make use of the following chain rule identity for $J(t)$, which follows from a direct computation: for any $p>0$,
\begin{equation}\label{lem1} 
    J(t)[|z|^p z] = \tfrac{p+2}{2}|z|^p [J(t)z] - \tfrac{p}{2}|z|^{p-2}z^2 [\overline{J(t)z}].   
\end{equation}

We will also utilize the following elementary estimate several times below: for $0\leq \alpha\leq 1$, 
\begin{equation}\label{simple-bound}
|e^{ix} - 1| \lesssim |x|^\alpha.
\end{equation}

\subsection{Well-posedness}\label{S:LWP} In this section we discuss the $H^1$ and $H^{1,1}$ well-posedness for \eqref{dmnls} (see e.g. \cite{AntSauSpa, CHL} for other well-posedness results for dispersion-managed NLS).  As much of what follows is standard, we focus only on the main points and the new estimates needed to treat the specific model \eqref{dmnls}; we refer the reader to \cite{Caz} for a general introduction to well-posedness for nonlinear Schr\"odinger equations.

  We construct solutions to \eqref{dmnls} as solutions to the following Duhamel formula:
\[
u(t) = e^{it\Delta}u_0 -i\int_1^t \int_0^1 e^{i(t-s)\Delta}\bigl[e^{-iD(\tau)\Delta}\{|e^{iD(\tau)\Delta}u(s)|^2e^{iD(\tau)\Delta}u(s)\}\bigr]\,d\tau\,ds. 
\]
We apply the standard contraction mapping argument, with the key nonlinear estimate given as follows:  By the chain rule, the Sobolev embedding $H^1(\R)\hookrightarrow L^\infty(\R)$, and the fact that $e^{i\cdot\Delta}$ is unitary on $H^1$, we have
\begin{align*}
\|&\int_0^1 \langle\partial_x\rangle\bigl[e^{-iD(\tau)\Delta}\{|e^{iD(\tau)\Delta}u|^2e^{iD(\tau)\Delta}u\}\bigr]\,d\tau \|_{L_t^1L_x^2([0,T]\times\R)} \\
& \lesssim \sup_{\tau\in[0,1]}T \| \langle \partial_x\rangle\bigl[ e^{-iD(\tau)\Delta}\{|e^{iD(\tau)\Delta}u|^2e^{iD(\tau)\Delta}u\}\bigr]\|_{L_t^\infty L_x^2} \\
& \lesssim \sup_{\tau\in[0,1]}T\|e^{iD(\tau)\Delta}u\|_{L_{t,x}^\infty}^2 \|\langle \partial_x \rangle u \|_{L_t^\infty L_x^2} \lesssim T \|\langle\partial_x\rangle u\|_{L_t^\infty L_x^2}^3.
\end{align*}
This allows us to establish local existence for $u_0\in H^1$ for times $T\sim \|u_0\|_{H^1}^{-2}$.  Similarly, by utilizing the chain rule for $J(t)$ (see \eqref{lem1}), we may establish local existence for $u_0\in H^{1,1}$, again with $T\sim \|u_0\|_{H^1}^{-2}$, albeit only with the crude bound
\[
\|x u(t)\|_{L^2}\lesssim \|J(t)u(t)\|_{L^2}+t \|\partial_x u(t)\|_{L^2} \lesssim (1+t)\|u_0\|_{H^{1,1}}.
\]

This $H^1$-subcritical well-posedness also includes the usual blowup alternative, that is, either $u$ is forward-global or there exists $T_*<\infty$ such that
\[
\lim_{t\to T_*}\|u(t)\|_{H^1}=\infty. 
\]
We will eventually obtain explicit bounds on the growth of the $H^1$-norm of solutions, which in particular imply that the solutions may be extended to be forward-global in time. 

\section{Proof of Main Result}\label{S:proof}
We let $u$ be a solution to \eqref{dmnls} with $\|u_0\|_{H^{1,1}}=\eps$.  By the local theory and Sobolev embedding, we may assume that
\begin{equation}\label{small-data-t1}
\|u(t)\|_{H^{1,1}}\leq 2\eps. 
\end{equation}
for $t\in[0,1]$, and that the solution exists on some maximal interval $(0,T_*)$ with $T_*>1$.

\subsection{Global existence and decay}\label{S:decay}  The first part of the proof of Theorem~\ref{thm1} is based on a bootstrap argument for times $t\geq 1$ involving the following `dispersive' and `energy' norms:
 \begin{equation}\label{XDXE} 
    \begin{aligned}
    \|u(t)\|_{X_D} &:= \|\hat{f}(t)\|_{L^\infty},\qtq{where} f(t)=e^{-it\Delta}u(t), \\
    \|u(t)\|_{X_E} &:= t^{-\frac{1}{20}}\bigl\{\|\jbrak{\partial_x} u(t)\|_{L^2} + \|J(t)u(t)\|_{L^2}\bigr\}.
    \end{aligned}
    \end{equation}
Note that by \eqref{J-id}, we may also write
\[
\|J(t)u(t)\|_{L^2}=\|xf(t)\|_{L^2}.
\]

We then define
\[
\|u(t)\|_{X} = \sup_{s\in[1,T]}\{\|u(s)\|_{X_D} + \|u(s)\|_{X_E}\}.
\]

Control over these norms will imply the desired $L^\infty$ decay.
\begin{lemma}\label{deduce-pointwise} For any $t\geq 1$, 
\[
\|u(t)\|_{L^\infty} \lesssim t^{-\frac12}\{\|u(t)\|_{X_D}+\|u(t)\|_{X_E}\}.
\]
\end{lemma}

\begin{proof} We write $f(t)=e^{-it\Delta}u(t)$ as above.  By \eqref{MDFM}, Hausdorff--Young, and Cauchy--Schwarz, we have
\begin{align*}
\|u(t)\|_{L^\infty} &= \|M(t)D(t)\F M(t)f(t)\|_{L^\infty} \\
& \lesssim t^{-\frac12}\{\|\hat f(t)\|_{L^\infty}+\|\F[M(t)-1]f(t)\|_{L^\infty}\} \\
& \lesssim t^{-\frac12}\{\|u(t)\|_{X_D} + t^{-\frac{1}{20}}\| |x|^{\frac{1}{10}} f(t)\|_{L^1}\} \\
& \lesssim t^{-\frac12}\{\|u(t)\|_{X_D}+ t^{-\frac{1}{20}} \|\langle x\rangle f(t)\|_{L^2}\}\\
 &\lesssim t^{-\frac12}\{\|u(t)\|_{X_D}+\|u(t)\|_{X_E}\}. 
\end{align*}
\end{proof}

The first part of Theorem~\ref{thm1} will follow from the following bootstrap estimate.

\begin{proposition}\label{P:bootstrap} Let $u:[1,T]\times\R\to\C$ be a solution to \eqref{dmnls} satisfying \eqref{small-data-t1}.  Then there exists $C>0$ (independent of $T$) so that  
    \begin{align*}
        \|u(t)\|_{X} \leq 8\eps + C\|u(t)\|^3_{X}        
    \end{align*} 
for all $t\in[1,T]$. 
\end{proposition}

We split this proposition into two lemmas. We begin by estimating the dispersive norm. 

\begin{lemma}[Dispersive bound] \label{disp} For any $t\geq 1$, 
    \begin{align}
        \|u(t)\|_{X_D} \leq 2\eps + C\int_1^t s^{-1-\frac{1}{20}}\|u(s)\|_X^3\,ds.
    \end{align}
\end{lemma}

\begin{proof}[Proof of Lemma~\ref{disp}]  We begin with the Duhamel formula for the profile $f(t)=e^{-it\Delta}u(t)$: 
\begin{align}
    f(t) = f(1)- i \int_1^t \int_0^1 e^{-is\Delta}e^{iD(\tau)\Delta}F(e^{iD(\tau)\Delta}u(s))\,d\tau\,ds,
\end{align}
where $F(z) = |z|^2 z$. Taking the Fourier transform yields the following:
\begin{multline}
    \hat{f}(t) = \hat{f}(1) - i(2\pi)^{-1}\int_1^t \int_0^1 \iint \bigl[e^{i(s+D(\tau))(\xi^2 - (\xi- \eta)^2 + (\eta - \sigma)^2 - \sigma^2)}\hfill \\ \times \hat{f}(s, \xi - \eta)  \hat{\bar{f}}(s, \eta - \sigma) \hat{f}(s,\sigma) \bigr]\,d\eta \,d\sigma \,d\tau \,ds
\end{multline}
Changing variables via $\xi-\sigma\mapsto \sigma$, we find that 
    \begin{align}
        \hat{f}(t) &= \hat{f}(1) - i (2\pi)^{-1} \int_1^t \int_0^1 \iint e^{2i(s+D(\tau))\eta\sigma} {G(s, \xi, \eta, \sigma)}\,d\sigma \,d\eta \,d\tau \,ds,
    \end{align}
    where
\[
G(s,\xi,\eta,\sigma):=\hat{f}(\xi-\eta) \hat{f}(\eta-\xi+\sigma)\hat{f}(\xi - \sigma).
\]

By Plancherel, we may write 
    \begin{align}
   	    \hat{f}(t) = \hat{f}(1) - i (2\pi)^{-1}\int_1^t \int_0^1 \iint \F_{\eta, \sigma} \left[e^{2i(s+ D(\tau))\eta\sigma}\right] \inv{\F}_{\eta, \sigma} \left[G(s, \xi, \eta, \sigma)\right] \,d\sigma \,d\eta \,d\tau \,ds
    \end{align}
Noting the identity 
\begin{align}
    \F_{\eta, \sigma} \left[e^{2i(s+D(\tau))\eta\sigma}\right] = \frac{1}{2(s+D(\tau))} e^{\frac{-i\eta\sigma}{2(s+D(\tau))}}
\end{align}
and the fact that 
\[
G(s, \xi, 0, 0) = |\hat{f}(s, \xi)|^2 \hat{f}(s, \xi),
\]
we may therefore write 
\begin{multline*}
    \hat{f}(t, \xi) = \hat{f}(1, \xi) - i(2\pi)^{-1}\int_1^t\int_0^1 \frac{1}{2(s+D(\tau))}|\hat{f}(s, \xi)|^2 \hat{f}(s, \xi) \,d\tau \,ds \\ \hfill + \int_1^t \int_0^1 \frac{1}{2(s+D(\tau))}\left[\iint \left[e^{\frac{-i\eta\sigma}{2(s+D(\tau))}}-1\right] \inv{\F}_{\eta, \sigma}[G] \,d\sigma\,d\eta\right]\,d\tau \,ds. 
\end{multline*}

In particular, this implies
\begin{align}\label{5}
i \partial_t \hat{f}(t,\xi) = \int_0^1 \frac{1}{2(t+D(\tau))}|\hat{f}(t, \xi)|^2 \hat{f}(t,\xi)\,d\tau + \int_0^1 \frac{1}{2(t+D(\tau))} \mathcal{R}(t,\tau,\xi) \,d\tau,
\end{align}
where 
\begin{align*}
    \mathcal{R}(t,\tau,\xi) = (2\pi)^{-1}\iint \left[e^{-\frac{i\eta\sigma}{2(t+D(\tau))}} - 1\right] \inv{\F}_{\eta,\sigma}[G] \,d\sigma\,d\eta.
\end{align*}

We now employ an integrating factor to remove the first term on the right-hand side of (\ref{5}).  With 
\begin{equation}\label{integrating-factor}
 \Theta(t) = \int_1^t\int_0^1 \frac{1}{2(s+D(\tau))}|\hat{f}(s, \xi)|^2 \,d\tau\,ds    
\end{equation}
and $g = e^{i\Theta(t)}\hat{f}$, we obtain 
\[
i\partial_t g = e^{i\Theta(t)}\int_0^1 \frac{1}{2(t+D(\tau))}\mathcal{R}(t, \xi)\,d\tau.   
\]

We want estimate this quantity in $L^\infty_\xi$. Using the definition of $\mathcal{R}$ from above, Cauchy--Schwarz and the bound \eqref{simple-bound}, we find that
\begin{align} \label{6}
    |\partial_t g| \lesssim \int_0^1 \iint |t+D(\tau)|^{-1-\frac15} |\eta|^{\frac15} |\sigma|^{\frac15} |\inv{\F}_{\eta,\sigma}[G](s, \xi,\eta,\sigma)| \,d\sigma\,d\eta\,d\tau
\end{align}

To proceed, we now need to invert the Fourier transform appearing in (\ref{6}).  Writing
\[
    \inv{\F}_{\eta, \sigma}[G] = \inv{\F}_\sigma \left[\inv{\F}_\eta \left[\hat{f}(\xi-\eta) \hat{f}(\eta - \xi + \sigma)\right]\hat{f}(\xi - \sigma)\right],    
\]
a direct computation leads to the estimate 
\begin{align}
    |\inv{\F}_{\eta, \sigma}[G](\eta,\sigma)| \lesssim \int |f(z-\eta)|\, |f(z)|\, |f(z-\sigma)| \,dz.
\end{align}
Thus, by the triangle inequality and Cauchy--Schwarz, and the fact that $D(\tau)\geq 0$\footnote{The fact that $D(\tau)\geq 0$ is a convenient consequence of our particular choice of parameters above.  To treat situations in which $D(\cdot)$ may take negative values, one only needs to observe that $\sup_{\tau\in[0,1]}|D(\tau)|\leq T_0$ for some $T_0>0$ and then begin the bootstrap estimate at times $t\geq 2T_0$, for example.}\label{the-footnote},
\begin{align*} 
    |(\ref{6})| & \lesssim \int_0^1 \iiint |t+D(\tau)|^{-1-\frac15} \left[|z-\eta|^{\frac15} + |z|^{\frac15}\right] \left[|z-\sigma|^{\frac15} + |z|^{\frac15}\right] \\
    &\quad\quad\quad\quad\quad \times  |f(z-\eta)||f(z)||f(z-\eta)| \,dz\,d\sigma\,d\eta \,d\tau \\
    & \lesssim t^{-1-\frac15}\|f\|_{L^1} \|\jbrak{x}^{\frac25}f\|^2_{L^1} \\
    &\lesssim t^{-1-\frac15}\|\jbrak{x}f\|^3_{L^2} \lesssim t^{-1-\frac{1}{20}}\|u(t)\|^3_{X}.
\end{align*}
The desired estimate then follows from the fundamental theorem of calculus and the triangle inequality:
\begin{align*}
    |g(t)| & = |\hat{f}(t)| \leq |\hat{f}(1)| + \int_1^t |\partial_s g(s)| \,ds \\
   & \implies \|u(t)\|_{X_D} \leq 2\eps + C\int_1^t s^{-1-\frac{1}{20}} \|u(s)\|^3_X \,ds.
\end{align*}
\end{proof}

We next estimate the energy norm. 

\begin{lemma}[Energy bound]\label{energy}
    For $u$ as above, we have the following estimate:
    \begin{align}
        \|u(t)\|_{X_E} \leq 2t^{-\frac{1}{20}}\eps + Ct^{-\frac{1}{20}} \int_1^t s^{-1+\frac{1}{20}}\|u(s)\|_X^3\,ds.
    \end{align}
\end{lemma}

\begin{proof}[Proof of Lemma~\ref{energy}]
    The starting point is the Duhamel formula
\begin{align*}
    u(t) = e^{i(t-1)\Delta} u(1) - i \int_1^t \int_0^1 e^{i(t-s)\Delta}e^{-iD(\tau)\Delta}F(e^{iD(\tau)\Delta}u(s))\,d\tau\,ds
\end{align*}
with $F(z) = |z|^2 z$.

We first estimate the weighted component of the $X_E$-norm.  Using \eqref{J-id}, we deduce 
\begin{align*}
    J(t)u(t) &= e^{it\Delta}xf(1) \\
    &\quad - i\int_0^t \int_0^1 e^{i(t-s)\Delta}e^{-iD(\tau)\Delta} \left[J(s+D(\tau))F(e^{iD(\tau)\Delta}u(s))\right]\,d\tau \,ds.
\end{align*}
Thus, using \eqref{lem1}, \eqref{small-data-t1}, \eqref{J-id}, Sobolev embedding, and the unitarity of $e^{i\cdot\Delta}$, we obtain
    \begin{align}
        \|{J(t)u(t)}\|_{L^2} &\leq 2\eps + C\int_1^t\int_0^1\bigl\| e^{iD(\tau)\Delta}u(s)\|_{L^\infty}^2 \bigl\| J(s+D(\tau))e^{iD(\tau)\Delta}u(s)\bigr\|_{L^2}\,d\tau\,ds \nonumber \\
        &\leq 2\eps + C\int_1^t \int_0^1 \bigl\| e^{iD(\tau) \Delta}u(s)\bigr\|^2_{L^\infty_x} \bigl\| e^{iD(\tau)\Delta} J(s)u(s)\bigr\|_{L^2}\,d\tau\,ds \nonumber \\
        &\leq 2\eps + C\int_1^t  s^{\frac{1}{20}} \biggl[\int_0^1 \bigl\| e^{i(s+D(\tau))\Delta}f(s)\bigr\|^2_{L^\infty_x}\,d\tau\biggr]\|u(s)\|_{X_E}\,ds.\label{energy-pause}
    \end{align}

To proceed, we use the factorization identity \eqref{MDFM} and estimate as we did in the proof of Lemma~\ref{deduce-pointwise}.  This yields
\begin{equation}\label{new-estimate}
\begin{aligned}
    \| &e^{i(s+D(\tau))\Delta}f(s) \|_{\lspace{\infty}{x}} \\
    &\lesssim |s+D(\tau)|^{-\frac12}\bigl\{ \| \hat{f}\|_{\lspace{\infty}{}} + \|\F[\mathcal{M}(s+D(\tau))-1]f \|_{\lspace{\infty}{}}\bigr\} \\
    &\lesssim |s+D(\tau)|^{-\frac12}\bigl\{\|\hat{f}\|_{\lspace{\infty}{}} + \|[\mathcal{M}(s+D(\tau))-1]f \|_{\lspace{1}{}}\bigr\}\\
    &\lesssim|s+D(\tau)|^{-\frac12}\bigl\{\|\hat{f}\|_{\lspace{\infty}{}} +|s+D(\tau)|^{-\frac15}\||x|^{\frac25}f\|_{\lspace{1}{}}\bigr\}\\
    &\lesssim |s+D(\tau)|^{-\frac12}\|\hat{f}\|_{\lspace{\infty}{}} + |s+D(\tau)|^{-\frac12-\frac15}\|\jbrak{x} f\|_{\lspace{2}{}} \\
    &\lesssim \bigl\{|s+D(\tau)|^{-\frac12}+|s+D(\tau)|^{-\frac12-\frac{3}{20}}\bigr\}\|u(s)\|_X \lesssim |s|^{-\frac12} \|u(s)\|_X,
\end{aligned}
\end{equation}
where we have again used the fact that $D(\tau)\geq 0$ (cf. the footnote on page \pageref{the-footnote}). Inserting this into \eqref{energy-pause}, we obtain the desired estimate for the weighted component of the $X_E$-norm. 

For the $H^1$ component of the $X_E$-norm, we estimate in much the same way, using the chain rule directly in place of \eqref{lem1}.
\end{proof}

With Lemma~\ref{disp} and Lemma~\ref{energy} in place, we readily obtain the estimate appearing in Proposition~\ref{P:bootstrap}.  Using a standard continuity argument, the well-posedness theory discussed in Section~\ref{S:LWP}, and Lemma~\ref{deduce-pointwise}, we deduce the the first part of Theorem~\ref{thm1}.  In particular, we have the following:

\begin{corollary}\label{corollary-first-part} For $\|u_0\|_{H^{1,1}}=\eps$ sufficiently small, there exists a unique, forward-global solution $u\in C_t H_x^{1,1}([0,\infty)\times\R)$ to \eqref{dmnls} with $u|_{t=0}=u_0$ obeying
\[
\|u(t)\|_X \lesssim \eps \qtq{for all} t\geq 1.
\]
In particular,
\[
\|u(t)\|_{L^\infty} \lesssim \eps(1+|t|)^{-\frac12}\qtq{for all}t\geq 0. 
\]
\end{corollary}

\subsection{Asymptotic behavior}\label{S:scattering} We now turn to the second part of Theorem~\ref{thm1}, namely, the asymptotic behavior of small solutions to \eqref{dmnls}.

\begin{proof}[Proof of \eqref{u-asymptotic}] To begin, we return to the setting of the proof of Lemma~\ref{disp}, this time with the bounds provided by Corollary~\ref{corollary-first-part} in hand.  In particular, recalling
\[
g(t) = e^{i\Theta(t)}\hat f(t),\qtq{with} \Theta(t) = \int_1^t\int_0^1 \frac{1}{2(s+D(\tau))}|\hat{f}(s, \xi)|^2 \,d\tau\,ds, 
\] 
the estimate of \eqref{6} now yields the bound 
\[
\|\partial_t g\|_{L_\xi^\infty} \lesssim t^{-1-\frac1{20}}\eps^3. 
\]
It follows that 
\[
\|g(t)-W_0\|_{L^\infty} \lesssim \eps^3 t^{-\frac{1}{20}}
\]
for some $W_0\in L^\infty$.  In particular, we have $|\hat f(t)|\to |W_0|$ in $L^\infty$, with the same rate of convergence.  

We next observe that
\[
\biggl| \int_0^1 \frac{1}{2(s+D(\tau))}\,d\tau - \frac{1}{2s} \biggr| \lesssim s^{-2}
\]
for all $s\geq 1$.  Using this, we deduce that 
\[
\Theta(t) = \tfrac12 |W_0|^2\log t + \Phi(t),
\]
where $\Phi(t)$ converges to a real-valued limit $\Phi_\infty$ in $L^\infty$ (with a rate of $t^{-\frac1{10}}$).  Thus, setting $W=e^{-i\Phi_\infty}W_0$, we may obtain 
\begin{equation}\label{f-hat-asymptotic}
\hat f(t)  = e^{-i\frac12|W_0|^2\log t} e^{-i\Phi_\infty} W_0 + \mathcal{O}(t^{-\frac1{20}}) = e^{-\frac12|W|^2\log t}W + \mathcal{O}(t^{-\frac{1}{20}})
\end{equation}
in $L^\infty$ as $t\to\infty$.  Finally, using \eqref{MDFM} and estimating as we did for Lemma~\ref{deduce-pointwise}, we obtain 
\[
u(t) = e^{it\Delta}f(t) = \mathcal{M}(t)\mathcal{D}(t)\hat f(t) + \mathcal{O}(t^{-\frac1{20}}). 
\]
Inserting the asymptotic behavior for $\hat f$ obtained in \eqref{f-hat-asymptotic}, we obtain the desired asymptotic behavior for $u(t)$.\end{proof}

\end{document}